\documentclass[11pt,letterpaper]{article}

\usepackage{amssymb}
\usepackage{amsthm}
\usepackage{amsmath}
\usepackage[pdftex]{graphicx}
\usepackage{graphicx}
\usepackage{fullpage}
\usepackage[margin=1in]{geometry}
\usepackage{enumerate}% http://ctan.org/pkg/enumerate
\usepackage{algorithm,algorithmic}

\usepackage[colorlinks=red,linkcolor=red]{hyperref}
\makeatletter
\renewcommand*{\eqref}[1]{%
  \hyperref[{#1}]{\textup{\tagform@{\ref*{#1}}}}%
}
\makeatother

\newcommand{\mcj}{\cj}

\DeclareMathOperator{\soc}{\mathbb{L}}

\usepackage{color}

\newtheorem{theorem}{Theorem}[section]
\newtheorem{lemma}[theorem]{Lemma}

\theoremstyle{definition}

\theoremstyle{remark}

\newcommand{\ca}{\mathcal{A}}
\newcommand{\cc}{\mathcal{K}}
\newcommand{\cj}{\mathcal{J}}
\newcommand{\cq}{S_{\mathcal{J}}}

\newcommand{\sfT}{{\sf T}}

\newcommand{\R}{\mathbb{R}}

\newcommand{\ba}{{a}}
\newcommand{\bb}{{b}}
\newcommand{\bc}{{c}}
\newcommand{\be}{{e}}

\newcommand{\bw}{{w}}
\newcommand{\bx}{{x}}
\newcommand{\by}{{y}}
\newcommand{\bz}{{z}}

\newcommand{\la}{\langle}
\newcommand{\ra}{\rangle}
\newcommand{\lj}{\langle}
\newcommand{\rj}{\rangle}
\newcommand{\tb}{b_k}

\newcommand{\lam}{\lambda}

\newcommand{\bi}{\begin{itemize}}
\newcommand{\ei}{\end{itemize}}
\newcommand{\beq}{\begin{equation}}
\newcommand{\eeq}{\end{equation}}
\newcommand{\itr}{{\rm int}}
\newcommand{\ben}{\begin{enumerate}}
\newcommand{\een}{\end{enumerate}}

\newcommand{\mm}{m}
\newcommand{\nn}{n}

\title{Multiplicative updates  for  symmetric-cone factorizations} 
\author{
	Yong Sheng Soh\thanks{Y.~S.~Soh (email: \url{matsys@nus.edu.sg}) is with the Department of Mathematics, National University of Singapore and the Institute of High Performance Computing, Agency for Science, Technology and Research.}
	\and Antonios Varvitsiotis\thanks{A.~Varvitsiotis (email: \url{avarvits@gmail.com}) is with the Engineering Systems and Design Pillar, Singapore University of Technology and Design.}}

\begin{document}

\maketitle
\begin{abstract}
Given a matrix $X\in \R^{m\times n}_+$ with non-negative entries, the cone factorization problem  over a  cone $\cc\subseteq \R^k$ concerns computing  $\{ \ba_1,\ldots, \ba_{m} \} \subseteq \cc$   and $\{ \bb_1,\ldots, \bb_{n} \} \subseteq~\cc^*$  belonging to its dual so that $X_{ij} = \langle \ba_i, \bb_j \rangle$ for all $i\in [m], j\in [n]$.  Cone factorizations are fundamental to mathematical optimization as they allow us to express convex bodies as feasible regions of linear conic programs.  In this paper, we introduce and analyze the symmetric-cone  multiplicative update (SCMU) algorithm for computing cone factorizations when $\cc$ is symmetric; i.e., it is self-dual and homogeneous.  Symmetric cones are of central interest in mathematical optimization as they provide a common language for studying linear optimization over the nonnegative orthant (linear programs), over the second-order cone (second order cone programs), and over the cone of positive semidefinite matrices (semidefinite programs).  The SCMU algorithm is multiplicative in the sense that the iterates are updated by applying a meticulously 
%We show that the appropriate generalization of positive scaling of elements in the non-negative orthant is a specifically
 chosen automorphism of the cone computed using a generalization of the   geometric mean to symmetric cones.  Using an extension of Lieb's concavity theorem  and von Neumann's trace inequality to symmetric cones, we show that the squared loss objective is non-decreasing along the trajectories of the SCMU algorithm.
%Our algorithm is mutiplicative in the sense that each update  is computed  by acting on the past iterate with  a meticoulously chosen automorphism of the cone.  As a consequence, membership of the iterates in the cone $\cc$ is automatic. 
Specialized to the nonnegative orthant, the SCMU algorithm  corresponds to the seminal algorithm by Lee and Seung for computing Nonnegative Matrix Factorizations. 

\end{abstract}

%\tableofcontents

\section{Introduction}

%A fundamental problem in mathematical optimization is to maximize linear functions over convex constraint sets.  An important consideration is that these constraint sets admit compact descriptions that facilitate the development of tractable numerical algorithms for solving these problems.  One powerful paradigm for obtaining compact descriptions of convex constrained sets is based on \emph{lifts} of \emph{affine slices} of structured convex cones.  More precisely, given a convex set $P$, an {\em extended formulation} or {\em lift} of $P$ over a (full-dimensional closed) convex cone $\cc$ is a description of $P$ as the projection of an affine slice of the cone $\cc$; i.e., $ P = \pi ( \cc \cap \mathcal{A})$, where $\pi$ is a linear projection and $\mathcal{A}$ an affine subspace.  Given such a $\cc$-lift, the problem of maximizing linear functions over $P$ reduces to one of solving a {\em linear conic program} where the cone is $\cc$.  

A fundamental problem in mathematical optimization is to maximize a linear function over a convex constraint set.  An important consideration for the development of tractable numerical algorithms for such problems is that the  constraint set admits a compact description.  A powerful paradigm for obtaining compact descriptions of convex  sets is based on \emph{lifts}  over structured convex cones.  More precisely, given a convex set $P$, an {\em extended formulation} or {lift} of $P$ over a (full-dimensional closed) convex cone $\cc$ is a description of $P$ as the projection of an affine slice of the cone $\cc$; i.e., $ P = \pi ( \cc \cap \mathcal{A})$, where $\pi$ is a linear projection and $\mathcal{A}$ an affine subspace.  

Given  a $\cc$-lift, the problem of maximizing linear functions over $P$ reduces to one of solving a {\em linear conic program} over the cone  $\cc$.  Linear conic programs (LCPs) capture many important classes of optimization problems -- for instance, LCPs where the cone $\cc=\R^k_+$ is the $k$-dimensional nonnegative orthant correspond to Linear Programs (LPs), LCPs where $\cc=\mathbb{S}^k_+$ is the cone of $k\times k$ positive semidefinite matrices correspond to  Semidefinite Programs (SDPs), and linear conic programs where $\cc=\soc_k=\{(x,t)\in \mathbb{R}^{k}\times \R: \|x\|_2\le t\} $ is the $k$-dimensional second order cone correspond to  Second Order Cone programs (SOCPs).  We often refer to extended formulations over $\R^k_+, \mathbb{S}^k_+,$ $\soc_k$ as LP, SDP, and SOCP-lifts respectively.

Lifted descriptions are powerful tool for optimization because there are many examples of convex  sets whose descriptions (at least on the surface) appear to be complex (such as by being specified by a number of inequalities that is exponential in the ambient dimensions) but do in fact admit compact lifted descriptions.  Examples of such constraint sets include the  unit ball of the $\ell_1$ norm, the spanning tree polytopes~\cite{martin,wong}, and the permutahedra~\cite{goemans}.  

The choice of cone $\cc$ also matters -- the smallest known LP-lift for the stable set polytope of a perfect graph with $n$ vertices  is $n^{O(\log n)}$ \cite{yannakakis} whereas there exists an SDP-lift of size $n+1$ using the theta body \cite{SLG}, and in fact it is the latter representation that gives rise to the {\em only known} polynomial-time algorithm for finding the largest stable set in a perfect graph. 
%\R^k_+
%Consequently, extended 
Moreover, the dimension of the cone $\cc$  %captures %broadly speaking,
is fundamentally linked to the computational complexity required to solve the associated linear conic program.  As such, for an extended formulation  to be  practically useful for optimization, 
%the cone $\cc$  should admit an efficient algorithm for linear optimization over its affine slices.     
it is important to seek extended formulations  involving cones of low dimensionality. %residing in small dimensions. %as small as possible. 
%For example, in the context of LP-lifts, there are important examples of polytopes that are defined as the convex combination of exponentially many (in the ambient dimension) points, or as the intersection of exponentially many linear inequalities, but nevertheless admit compact linear lifts.  
%One of the key motivations for studying cone factorizations is its connection to the fundamental concept of computational complexity in optimization problems.  More concretely, \emph{conic programming} is the class of optimization instances in which we minimize a linear objective over a constraint set specified as the affine slice of a closed convex cone.  The dimension of the cone, broadly speaking, is representative of the computational complexity of the associated conic program.  As such, given a constraint set, it is of practical importance to seek descriptions using cones residing in dimensions as small as possible.  To this end, Yannakakis drew the connection between compact descriptions and NMFs of appropriately defined matrices~\cite{yannakakis}.  

Despite the usefulness of extended formulations, it is not immediately clear how one systematically searches for such descriptions.  To this end, Yannakakis shows that LP-lifts of polytopes are fundamentally linked to the existence of a structured matrix factorization of a slack matrix~\cite{yannakakis}, while Gouveia, Parrilo, and Thomas \cite{lifts} extend this connection to lifts using more general cones.  

We explain this connection more precisely: Given an entrywise non-negative matrix $X \in \mathbb{R}_+^{\nn \times \mm}$  and a (full-dimensional convex) cone $\cc$ lying in inner product space $(V, \la \cdot, \cdot \ra)$, the {\em cone  factorization problem} concerns finding two collections of vectors $ \ba_1,\ldots, \ba_{\nn} $ in the cone $ \cc $ and $ \bb_1,\ldots, \bb_{\mm} $ in the dual cone $\cc^*$~where
\begin{equation}\label{conefactorization}
X_{ij} =\langle \ba_i, \bb_j \rangle, \text{ for all } i\in [\nn], j\in [\mm].
\end{equation}
% Given a matrix $X \in \mathbb{R}_+^{\mm \times \nn}$ with non-negative entries, the {\em cone  factorization} problem is concerned with computing elements $ \ba_1,\ldots, \ba_{\mm} $ belonging to the cone $ \cc $ and $ \bb_1,\ldots, \bb_{\nn} $ belonging to its dual $\cc^*$ such that
%\begin{equation}\label{conefactorization}
%X_{ij} =\langle \ba_i, \bb_j \rangle, \text{ for all } i\in [\mm], j\in [\nn].
%\end{equation}
In the case where $\cc=\R^k_+$, the cone factorization problem reduces to  the {\em Non-negative Matrix  Factorization} (NMF) problem in which we seek a collection of non-negative vectors ${a_1,\ldots, a_m\in\R^k_+}$ and $b_1,\ldots,b_n\in \R^k_+$ such that $X_{ij}= \la a_i, b_j\ra $ for all $ i\in [m]$ and $j\in [n]$.  The smallest $k\in \mathbb{N}$ for which $X$ has an $k$-dimensional NMF is the {\em nonnegative rank} of the matrix $X$. NMFs were initially studied within the field of linear algebra (see, e.g.,~\cite{cohen}), and later gained prominence as a dimensionality reduction tool providing interpretable parts-based representations  of non-negative data
%in applications where non-negativity considerations arises naturally 
 \cite{LSnature}.  
 
 In the context of LP-lifts, Yannakakis~\cite{yannakakis} showed that the existence of an $\R^k_+$-lift for a polytope $P=\{x\in \R^d: \la c_i, x\ra \le d_i, \ i\in [\ell]\}$ is equivalent to the existence of a $k$-dimensional  NMF  of its {\em slack matrix}~$S_P$,   a nonnegative rectangular matrix  
 where the rows are indexed by the bounding hyperplanes   $\{ x : \la c_i, x\ra \le d_i \}$ of~$P$,  the columns are indexed by the extreme points $v_j$, and the 
$ij$-entry is  given by   $S_{ij}=d_i- \la c_i, v_j\ra$.
%rows are indexed by the facets $c_i^\sfT x\le d_i$, its columns indexed by  extreme points $v_j$,  and the .    
%
%Specifically, let $P$ be a polytope.  The \emph{slack matrix} $S_P$ is defined by $S_{ij}=d_i- \la c_i,v_j \ra$, where the rows are indexed by facets of the form $\{ x : \la c_i, x\ra \le d_i \}$, and the columns are indexed by the extreme points $v_j$.  Yannakakis~\cite{yannakakis} showed that there exists a description 
%\begin{equation}\label{eq:extendedformulation}
%P = \pi ( \mathbb{R}^k_{+} \cap \mathcal{A}),
%\end{equation}
%for some linear projection map $\pi$ and some affine subspace $\mathcal{A}$, if and only if there exists a rank $k$ NMF of the slack matrix~$S_P$ -- here, a factorization of \emph{rank} $k$ is one where $X_{ij}= a_i^\sfT b_j$ with $\{a_i\}_{i=1}^{m},\{b_j\}_{j=1}^{n} \subset \mathbb{R}^{k}$.  Descriptions of the form \eqref{eq:extendedformulation} are known as \emph{extended formulations} or a {\em lift} of $P$ over the cone $\R^k_+$.  As such, the existence of a compact description of a polytope is effectively a statement about the \emph{non-negative rank} of its associated slack matrix.  Examples of polytopes whose number of vertices or facets grow exponentially in the dimension of the polytope, but do in fact admit compact extended formulations include the $\ell_1$ unit ball, the spanning tree polytopes \cite{martin,wong}, and the permutahedra~\cite{goemans}.
Subsequently, Gouveia, Parrilo, and Thomas \cite{lifts} showed  that a polytope $P$ admits an extended formulation over a (full-dimensional) closed convex cone $\cc$ if its slack matrix $P$ admits an exact cone factorization over $\cc$.  In fact, this relationship holds for convex sets beyond polyhedral ones, although it is then necessary to extend the notion of a slack matrix to that of a \emph{slack operator} indexed over the infinite set of extreme points and/or facets. 

% As such, the existence of lifts for cases where $\cc$ is the second order cone (or the Lorentz cone) or the cone of semidefinite matrices lead to descriptions of convex sets via Second Order Cone Programming (SOCP) or Semidefinite Programming (SDP) respectively.

Cone factorizations also   have  important applications beyond optimization.  One such prominent example is in quantum information science,  where factorizations over the PSD cone $\mathbb{S}^k_+$ are relevant to the quantum correlation generation problem \cite{qcorr} and  one-way quantum communication complexity \cite{sam}.  Specifically, given a non-negative matrix $X$, the {\em positive semidefinite matrix  factorization} (PSDMF) problem concerns computing two families of $k \times k$ PSD matrices $A_1,\ldots, A_{\mm}$ and $B_1,\ldots,B_{\nn}$ such that $X_{ij}= {\rm tr}(A_iB_j)$ for all $ i\in [\mm],\ j\in [\nn]$.  The smallest $k$ for which $X$ admits an $k$-dimensional PSDMF is known as the {\em PSD-rank} of $X$~\cite{psdrank}. 
%As already discussed, PSDMFs are relevant for optimization as they allow to reduce linear optimization over a polyhedron  $P$ to a semidefinite programming problem. Nevertheless, 
%PSDMFs and the PSD-rank of certain matrices have interesting connections to 
% [FURTHER ELABORATION]
%A different motivation for studying cone factorizations, and specifically in the setting where the cone $\cc=\mathbb{S}^k_+$ is the set of Positive Semidefinite (PSD) matrices, arises from the field of Quantum Information Theory. 

\paragraph{Computing cone factorizations.}  The task of computing an exact cone factorization is in general  intractable. Concretely,  in the specific setting where $\cc$ is the non-negative orthant, Vavasis showed that even computing  the non-negative rank is  NP-hard \cite{vavasis}.  On the positive side, Arora et~al. propose an algorithm for computing $k$-dimensional NMFs whose complexity is polynomial time in the dimensions of the matrix, provided the rank parameter $k$ is held constant~\cite{arora}.  

Despite the hardness of NMFs, a wide range of numerical algorithms for computing (approximate) NMFs %by seeking local minima 
have been developed and implemented in a wide range of data analytical applications.  Most, if not all, algorithmic approaches  are based on the principle of alternating minimization. %in which we fix one factor and we minimize over the other factor.
%  which reduces  to problem to 
  %a non-negative least squares problem at each iteration. 
 % subsequently leads to solving a convex program with polyhedral constraints as a sub-routine.
One of the most prominent approaches for computing NMFs is the Multiplicative Update (MU) algorithm proposed by Lee and Seung \cite{LS00}.  The update scheme is based on the majorization-minimization framework \cite{Lange}, and it operates by performing pointwise scaling by carefully selected non-negative weights.  In fact, it is  the simplicity of the MU scheme that drives its popularity.  Other alternative methods include variants of projected gradient methods and coordinate descent.

There are also a wide range of methods for computing PSDMFs, which    operate in a similar vein by  alternating minimization. % As such, one is required to solve a convex program subject to positive semidefinite constraints as a sub-routine.  
In particular, Vandaele et al. propose  algorithmic   approaches  based on the projected gradient method and coordinate descent \cite{VGG18}, while the authors \cite{LF20, LF20b, LLTF20} develop algorithms by drawing connections between PSDMFs to the affine rank minimization and the phase retrieval problems from the signal processing literature.  Recently, Soh and Varvitsiotis introduced the Matrix Multiplicative Update (MMU) method, which is the analogue of the MU scheme for computing PSDMFs in which updates are performed by congruence scaling with appropriately chosen PSD matrices \cite{YSV}.  In particular, the Matrix Multiplicative Update scheme retains the simplicity that the MU update scheme for NMF offers. % Additionally, the MMU method can be used for computing block-diagonal PSD factorizations and also PSD factorizations of nonnegative tensors. 

\paragraph{Contributions.}  In this work, we introduce the {\em symmetric-cone multiplicative update} (SCMU) algorithm    for computing cone factorizations \eqref{conefactorization} in the case where  the cone $\cc$ is \emph{symmetric}; that is, it is self dual and homogenous (i.e., the automorphism group of $\cc$ acts transitively on $\mathrm{int}(\cc)$).  The SCMU algorithm  corresponds to   Lee-Seung's Multiplicative update  algorithm \cite{LS00} for NMF when $\cc$ is the non-negative orthant, and to  the Matrix Multiplicative update  algorithm in \cite{YSV} for PSDMF when $\cc$ is the cone of PSD matrices.

%Linear optimization over symmetric cones is a problem  of fundamental interest in optimization as it provides a unifying framework for studying 
% linear programming, second-order cone programming, and semidefinite  programming, arguably some of the most important models of convex optimization. In view of  the link between cone factorizations and conic lifts \cite{lifts},  any algorithm for symmetric-cone factorizations gives rise to  a method  for reducing linear programming to linear conic programming over a symmetric cone.   
The SCMU  scheme is based on the principles of the majorization-minimization framework, and is {\em multiplicative} in the sense that  iterates are updated  by applying an appropriately chosen automorphism of (the interior) of the cone $\cc$.  
%Consequently, membership in the cone $\cc$ is automatically preserved throughout the execution of the algorithm.
% Furthermore, we show that the Euclidean square loss is nonincreasing along the trajectories of the algorithm, and fixed points  corresponds to first-order stationary points. 
As a result, the SCMU  algorithm ensures that iterates remain in the interior of the cone $\cc$, provided it is initialized in the interior.  
In terms of performance guarantees for the SCMU  algorithm,  we prove is that the Euclidean squared loss is non-increasing along its trajectories and moreover, we also show that fixed points our our scheme correspond to first-order stationary points.  Additionally,  if the starting factorization lies in a direct sum of simple symmetric cones,   the direct sum structure    is preserved throughout the execution of the SCMU algorithm. As such, the SCMU algorithm specifies a method for computing hybrid cone factorizations. %and thus the  SCMU  algorithm leads to method for computing hybrid cone factorizations. 

In terms of applications, when applied to computing cone factorizations of slack matrices, the SCMU algorithm  provides a practical way to compute conic lifts of convex sets -- for instance, it can be used to compute approximate SOCP lifts when $\cc$ is chosen to be (products of) the second order cone $\soc_k=\{(x,t)\in \mathbb{R}^{k}\times \R: \|x\|_2\le t\} $. In particular, we give explicit examples where we apply the SCMU algorithm for computing SOCP lifts of regular polygons.
%In the  case where $\cc$ is the nonnegative orthant our algorithm specializes to Lee-Seung's algorithm for NMF~\cite{LS00} whereas in the case where $\cc$ is the PSD cone, it  corresponds to the matrix multiplicative update  method from~\cite{YSV}. Furthermore, specialized to the case where $\cc$ is the $n$-dimensional second-order cone $\soc_n=\{(x,t)\in \mathbb{R}^{n}\times \R: \|x\|_2\le t\} $ our method leads  to an    algorithm for finding SOC factorizations, and consequently, for computing SOC-lifts. %Second-order cone programs have siginificant expressive power and can be solved more efficiently then 

\paragraph{Paper Organization.}  In Section 2, we provide necessary background material concerning Euclidean Jordan Algebras and Symmetric Cones.  In Section 3, we describe our approach and derive the multiplicative update algorithm for symmetric-cone factorizations.  In Section 4, we show that the Euclidean square loss is nonincreasing along the algorithms' trajectories and that fixed points satisfy the first-order optimality conditions. In Section 5, we conclude with numerical experiments focusing on lifts over the second-order cone.

\section{Euclidean Jordan Algebras and Symmetric Cones}

In this section, we provide brief background on symmetric cones to describe our algorithm and our analysis.  Our discussion requires the formal language of \emph{Euclidean Jordan algebras} (EJAs) from which symmetric cones arise.  For further details and omitted proofs, we refer the reader to~\cite{FK,vanden}.  %To avoid confusion in our subsequent discussion, we denote the Euclidean inner product  by $x^\sfT y = \sum x_i y_i$  and we use $\lj \cdot, \cdot \rj$ to denote the inner product associated to an EJA (we discuss this shortly) in the remainder of this paper.  

\medskip 

\noindent {\bf Jordan Algebras.}  Let $\cj$ be a finite-dimensional vector space endowed with a bilinear product  $\circ: \mcj\times \mcj \to \mcj$.  We say that the pair $(\cj,\circ)$ form  a \emph{Jordan algebra} if the following properties hold:
\begin{equation*}
\begin{aligned}
\bx \circ \by & = \by \circ \bx, \\
\bx^2 \circ (\bx \circ \by) & = \bx \circ (\bx^2\circ  \by). %\label{productprop}
\end{aligned}
\end{equation*}
%\item  $\bx^\sfT(\by \circ \bz)=(\bx\circ \by)^\sfT \bz$  (not standard, Vandenberghr includes this)
Here, we use the shorthand notation  $\bx^2=\bx\circ \bx$.  In the remainder of this paper, we {assume} that the Jordan algebra $(\cj,\circ)$ has an identity element; that is, there exists $\be\in \mcj$ such that $\be \circ \bx = \bx \circ {e} = \bx$.

\medskip

\noindent {\bf Euclidean Jordan Algebras.} 
%A Jordan algebra  $(\cj,\circ)$ over $\R$ is \emph{Euclidean} if 
%there exists a quadratic form $\langle \cdot , \cdot \rangle : \mcj\times \mcj \to\R_+ $ which is  symmetric, positive definite, and associative,~i.e., 
%\begin{align}
%& \la \bx,\bx\ra >0, \ \forall \bx\ne 0 \\
% & \la \bx,\by\ra=\la \by,\bx\ra \\
%& \langle \bx \circ \by , \bz \rangle = \langle \by, \bx \circ \bz \rangle, \label{associa}
%\end{align}
A Jordan algebra  $(\cj,\circ)$ over $\R$ 
 %We say that a Jordan algebra $(\cj,\circ)$ over $\R$ is \emph{Euclidean} 
 which is equipped with  an associative inner product $(\cdot,\cdot)$, (i.e., 
  %$a bilinear mapping which is symmetric, positive definite, and$
$( \bx \circ \by , \bz ) = ( \by , \bx \circ \bz ),$ for all $ x,y,z$)
%Recall that an \emph{inner product} over $\R$ is \emph{symmetric} (i.e. $\lj \bx,\by\rj =\lj \by,\bx\rj$) and \emph{positive definite} (i.e., $\lj \bx,\bx\rj \geq 0$ for all $x$ with equality if and only if $x=0$).
is called~\emph{Euclidean}.
\medskip 

\noindent {\bf Spectral decomposition and powers.}  Our algorithm requires the computation of square-roots as well as inverses of Jordan Algebra  elements.  To explain how these  are defined, we require a \emph{spectral decomposition theorem} for EJAs.  In what follows, we describe a version of this theorem based on \emph{Jordan frames} (see Type-II spectral decomposition~\cite[Theorem III.1.2]{FK}).  

More concretely, an {\em idempotent} is an element $x\in \cj$ satisfying $x^2=x$.  We say that an idempotent is {\em primitive} if it is nonzero and  cannot be written as a sum of two nonzero idempotents.  We say that a collection of primitive idempotents $\{ c_1, \ldots, c_m \} \subseteq \cj $ form a {\em Jordan frame} if they satisfy $(i)$  $\bc_i\circ \bc_j= 0$ if and only if $i \neq j$, and $ (ii)$ $\sum_i {c}_i = \be$.  It follows that distinct elements of a Jordan frame $\bc_i$ are orthogonal with respect to the inner product $( \cdot,\cdot )$.

For every $x\in \cj$ there exists a Jordan frame    $c_1,\ldots, c_r$, and real numbers $\lambda_1,\ldots, \lambda_r$~with
$x=\sum_{j=1}^r\lambda_jc_j.$
Here, the value $r$ is called the \emph{rank} of $x$.  The scalars $\{\lambda_1,\ldots, \lambda_r\}$ are called  the \emph{eigenvalues} of $x$ and -- up to re-ordering and accounting for multiplicities -- are uniquely specified. 
%
%The numbers $\lambda_j$ (and their multiplicities) are uniquely defined. % Furthermore, we have
%%$$a_k(x)=\sum_{1\le i_1\ldots \le  i_k\le r}\lambda_{i_1}\ldots\lambda_{i_k},$$
%%where the $a_k(x)$'s are the coefficients of the minimal polynomial of $x$ (recall \eqref{minimal}). In particular, 
%The {\em trace} of $x\in \cj$, is defined as the sum of its eigenvalues, i.e.,
%${\rm tr}(x)=\sum_{j=1}^r\lambda_j.$ 
%Let $(\cj,\circ, ( \cdot, \cdot ))$ be an EJA.  Given $x\in \cj$, there exists a Jordan frame $\{c_1,\ldots,c_r\}$ with a corresponding sequence of reals $\lambda_1,\ldots, \lambda_r$ such that $x=\sum_{j=1}^{r} \lambda_j c_j$.  

Given $\bx\in \cj$ with spectral decomposition $\bx = \sum_i \lambda_i {c}_i$, we define its $a$-th power as $\bx^a=\sum_i\lambda_i^a\bc_i$ whenever $\lambda_i^a$ exists for all $i$.  We say that an element $x\in \cj$ is {\em invertible} if all its eigenvalues are nonzero. 

%Using the spectral decomposition, given $\bx = \sum_i \lambda_i {c}_i \in \cj$ where $\lambda_i >0$, we define $\bx^a=\sum_i\lambda_i^a \bc_i$ for any $a$.  
%In particular, given $\bx = \sum_i \lambda_i {c}_i \in \cj$ where $\lambda_i \neq 0$, we define $\bx^{-1}=\sum_i\lambda_i^{-1} \bc_i$.  
%It follows from these definitions that $\bx^{1/2} \circ \bx^{1/2} = \bx$ and $\bx^{-1} \circ \bx = e$.
\medskip

\noindent{\bf The canonical trace inner product.}  In our description of an EJA so far, we have only assumed the existence of an inner product $(\cdot , \cdot )$.  There is in fact a canonical choice:  Given the spectral decomposition $x=\sum_{j=1}^{r} \lambda_j c_j$, define the {\em trace} to be the sum of the eigenvalues ${\rm tr}(x)=\sum_{j=1}^r \lambda_j$.  Then, the bilinear mapping ${\rm tr}(x\circ y)$ is positive definite,  symmetric, and also satisfies the associativity property (see, for instance, \cite[Proposition II.4.3]{FK}).
%One can check that the \emph{trace inner product} ${\rm tr}(x\circ y)$ is indeed symmetric, positive definite, and associative (see~\cite[Proposition II.4.3]{FK}). %As a consequence, for a Jordan algebra over $\R$ with a unit element, $\cj$ is Euclidean if and only if ${\rm tr}(x\circ y)
%$ is positive definite. 
Furthermore, in the case where the EJA is \emph{simple} (i.e., it cannot be expressed as the direct sum of smaller EJAs), the inner product  $(\cdot , \cdot )$ is  a positive scalar multiple of the canonical one  ${\rm tr}(x\circ y)$ (\cite[Proposition III.4.1]{FK}). 
In the remainder of this paper we use the notation $\lj x, y \rj={\rm tr}(x\circ y)$ to denote the canonical EJA  inner product.

\medskip

\noindent{\bf The Lyapunov transformation and the quadratic representation.} 
 Consider  an EJA $(\cj,\circ)$.  As the algebra product $\circ$ is bilinear, given any $\bx$, there is a matrix $L({\bx})$~satisfying
$ \bx \circ \by=L({\bx}) \by$ for all $\by \in \cj$.  We note that it is fairly common in the literature to express the defining properties of an EJA in terms of the Lyapunov operator $L(x)$. %Additionally, it follows that $P_x$ is symmetric with respect to $\la \cdot, \cdot \ra$, i.e., 
For isntance, the requirement $x^2\circ (\bx \circ \by)   = \bx \circ (  \bx^2 \circ \by )$ is equivalent to the requirement that the operator $L({\bx})$ commutes with $L({\bx^2})$, while the requirement $\langle \bx \circ \by , \bz \rangle = \langle \by, \bx \circ \bz \rangle$ is equivalent to the requirement $\la L(\bx) \by, \bz\ra=\la \by, L({\bx})\bz\ra$, i.e., the operator $L(\bx)$ is symmetric with respect to the trace inner product.

The \emph{quadratic representation} of $\bx\in \cj$ is defined by 
$P({\bx})= 2L^2({\bx}) - L({\bx^2})$.  The term quadratic alludes  to the fact that $P({\bx})\be=\bx^2$. 
%\beq
%\eeq
The quadratic representation satisfies the following  properties:
\begin{align}
&\la P(\bx) \by, \bz\ra=\la \by, P({\bx})\bz\ra \label{Pxsymmetric}\\
% &P({\bx})\be=\bx^2\label{quadratic} \\
 &P({\lambda \bx})=\lambda^2P(\bx) \label{scalar}\\
 & P({P({\bx})\by})=P({\bx}) P({\by})P({\bx}) \label{unfolding}\\ 
&(P({\bx}))^a=P({\bx^a}),  \text{ if } x^a \text{ is defined }\label{powers}\\
&P(\bx)(\bx^{-1})=\bx, \text{ if } \bx \text{ is invertible}\label{powers2}\\
&P(\bx)(\cq)=\cq \text{ and } P({\bx})(\itr(\cq)) =\itr(\cq),  \  \text{ for all invertible } \bx\in \cj \label{invariance}\\
& P(\bx)  \succeq 0, \text{ for any }\bx\in \cq\label{psdness}\\
& P(\bx)  \succ 0, \text{ for any }\bx\in {\rm int}(\cq)\label{pdness}
\end{align}

\noindent{\bf Cone of squares and symmetric cones.}  
Given an inner product space $(V, \phi(\cdot, \cdot))$, we say that a cone $\cc \subset V$ is {\em symmetric} if:
\begin{enumerate}
\item[$(i)$] $\cc$ is \emph{self-dual}; i.e., $\cc^*=\{y\in V: \phi(y,x)\ge 0 \ \forall x\in \cc\} = \cc$, and
\item[$(ii)$] $\cc$ is {\em homogeneous}; i.e., given $u,v \in~{\rm int}(\cc)$, there is an invertible linear map $T$ such that $T(u)=v$, and $T(\cc)=(\cc)$ (or $\mathrm{Aut}(\cc)$ acts transitively on ${\rm int}(\cc)$).
\end{enumerate}
Given an EJA $(\cj,\circ, ( \cdot, \cdot ))$, the cone of squares is defined as the set  
$$\cq=\{ \bx^2 : \bx \in \cj \}.$$ 
The set $\cq$ is self-dual with respect to the inner product $\la x,y\ra={\rm tr}(x\circ y)$, and hence closed and convex.
%One can show that this cone is closed and convex.  Specifically, this  follows from the fact that $\cq$ is self-dual with respect to the inner product $\la x,y\ra={\rm tr}(x\circ y)$, i.e., $\cq$ is equal to its dual cone
%$$\cq^*=\{ y\in \cj :\la y,x^2\ra\ge 0 \  \forall x\in \cj \}=\{y\in \cj : L(y)\succeq 0\}.$$ 

The key connection between symmetric cones %which are fundamental objects in optimization, 
and EJAs is that \emph{all} symmetric cones arise as cone of squares of some EJA; see e.g., \cite[Theorem III.3.1]{FK}.  In particular, the homogeneity property of the cone of squares follows from the existence of a \emph{scaling point}; i.e., for    $x,y\in {\rm int}(\cq)$ it~holds
\beq\label{scalingp}
P(w)x=y, \ \text{ where } \ w=P(\bx^{-1/2})(P({\bx}^{1/2})\by)^{1/2}.
\eeq
  As  both $x,y\in {\rm int}(\cq)$, it follows by \eqref{pdness} that $P(y^{1/2})$ and $ P(x^{-1/2})$ are positive definite. Consequently, the product  $P(w)=P(y^{1/2})P(x^{-1/2})$ is also positive definite and thus invertible.  By~\eqref{invariance}, $P(w)$  defines an automorphism of $\cq$.

\medskip

%\noindent{\bf Symmetric Cones vs. Cones of Squares.} Although not immediately apparent, a cone is symmetric if and only if it is the cone of squares of some Euclidean Jordan Algebra, . It is not  hard to see that the  cone of squares $\cq$ is a symmetric cone. Indeed, as~$\cq$  is self-dual  it remains to show homogeneity. 
%This follows from the following two properties \eqref{invariance} and \eqref{pdness}. 
%\begin{align} 
%& x\in \cj \text{ invertible implies that } P_x({\rm int}(\cq))= {\rm int}(\cq)) \label{sdva}\\
%&  x\in {\rm int}(\cj) \text{ implies that }P_x\succ0.\label{svdsdv}
%\end{align}

%Indeed, setting $z=P({\bx}^{1/2})\by$ and using \eqref{unfolding} we have 
%\beq\label{dsvsw}
%P(w)=P(P(\bx^{-1/2})z^{1/2})=P(\bx^{-1/2})P(z^{1/2})P(\bx^{-1/2}),
%\eeq
%and thus 
%$$P(w)x=P(\bx^{-1/2})P(z^{1/2})P(\bx^{-1/2})x=P(\bx^{-1/2})P(z^{1/2})e=P(\bx^{-1/2})z=y,$$
%where for the second and third equalities  we used   \eqref{quadratic}. 

%\medskip 

\noindent{\bf Classification of EJAs and Symmetric Cones.}  There is in fact a complete classification of all  EJAs.  Let $\mathbb{C}, \mathbb{H}, \mathbb{O}$ denote the fields of complex numbers, quaternions,  and octonions respectively.  Then every \emph{finite}-dimensional EJA $(\cj,\circ)$ is isomorphic to a direct sum of these \emph{simple} EJAs:
\begin{enumerate}
\item Symmetric matrices over $\R$ or Hermitian matrices over $\mathbb{C}$ or $\mathbb{H}$ with  $X \circ Y =~(XY+YX)/2$.
\item The space $ \R\times \R^k$ with product $\bx \circ \by = (x_0, \bx_1)\circ  (y_0, \by_1)= (\bx^\sfT y , x_0 \by_1 + y_0 \bx_1)$.
\item The space of 3-by-3 Hermitian matrices over $\mathbb{O}$ with product $X \circ Y = (XY+YX)/2$.
\end{enumerate}
Subsequently, since all symmetric cones arise as cones of squares of some EJA, it follows that any symmetric cone corresponds to   a direct sum of:
\begin{enumerate} 
  \item Symmetric PSD matrices over $\R$ or Hermitian PSD matrices over $\mathbb{C}$ %(complex values) 
  or $\mathbb{H}$%(quarternions).
  \item  Second-order cones, i.e., $\soc_d=\{(x,t)\in \R^d\times \R:\  \|x\|_2\le t\}$.
  \item 3-by-3 PSD matrices over octonions $\mathbb{O}$.
\end{enumerate} 
In Table \ref{fig:summary_EJA}, we summarize the basic information for the EJA corresponding to the cone of positive semidefinite matrices and the second order cone.

\medskip

\medskip 
\noindent {\bf The Geometric Mean.} Given a symmetric cone $\cc$, the {\em (metric) geometric mean} of two elements $\bx, \by\in \itr(\cc)$ is defined as the scaling point that takes $x^{-1}$ to $y$ \cite{sturm}, i.e.,  
% EJA $(\cj,\circ, \la \cdot, \cdot \ra)$ and let $\cq$ the  corresponding cone of squares (i.e., an arbitrary symmetric cone). 
\beq\label{cmean}
\bx\#\by=P(\bx^{1/2})(P({\bx}^{-1/2})\by)^{1/2}.
\eeq
As established in \eqref{scalingp}, the geometric mean satisfies:
\beq\label{cmean2}
P({\bx\#\by})\bx^{-1}=\by.
\eeq
%Indeed, setting  $\bz=P({\bx}^{-1/2})\by$ , we have that
% $$P({\bx\#\by})=P(P(x^{1/2})z^{1/2})=P({\bx^{1/2}})P({\bz^{1/2}})P({\bx^{1/2}}),$$
% where for the last equality we used property \eqref{unfolding}. Thus, we get  
% $$P({\bx\#\by})\bx^{-1}=P({\bx^{1/2}})P({\bz^{1/2}})P({\bx^{1/2}})\bx^{-1}=P({\bx^{1/2}})P({\bz^{1/2}})\be=P({\bx^{1/2}})\bz=\by,$$
% where  for the second equality we used that $P({\bx^{1/2}})\bx^{-1}=e$ (which follows  from \eqref{quadratic}) and for the last equality we used  the definition of $z$.
%%Note that we need $\bx$ to be invertible for $P_{\bx}^{-1/2}=P_{\bx^{-1/2}}$ to be well-defined, and $\by$ should be invertible as the roles of $\bx$ and $\by$ are interchangeable (we will see this shortly).
In fact, it turns out that the geometric mean is the unique element of ${\rm int}(\cc)$ satisfying~\eqref{cmean2}. Indeed, suppose $w\in{\rm int}(\cc)$ satisfies $P(w)\bx^{-1}=\by$.  We then have
\beq\label{svvdsf}
P(x^{-1/2})y=P(x^{-1/2})P(w)\bx^{-1}=P(x^{-1/2})P(w)P(x^{-1/2})e=P(P(x^{-1/2})w)e=(P(x^{-1/2})w)^2.
\eeq
As $P(x^{-1/2})w$ and $P(x^{-1/2})y$ are both elements of $\cc$ (recall \eqref{invariance}), the equality in \eqref{svvdsf} implies that $P(x^{-1/2})w=(P(x^{-1/2})y)^{1/2}$ and thus $w=P(\bx^{1/2})(P({\bx}^{-1/2})\by)^{1/2}=x\#y$.
%
%An equivalent way to define the   geometric mean $x\#y$  is as the unique element of the cone of squares $\cq$ satisfying 
%
%
%Indeed, let $\bx\#\by$ as defined  in  \eqref{cmean}, we show that \eqref{cmean2} holds. 
%
%
% 
% Conversely, let $\bz\in \cq$ such that $P_{\bz}\bx^{-1}=\by$, we need to show that $z=\bx\#\by$. Indeed,  we have~that 
% $$P_{\bx^{-1/2}}\by=P_{\bx^{-1/2}}P_{\bz}\bx^{-1}=P_{\bx^{-1/2}}P_{\bz}P_{\bx^{-1/2}}\be=P_{P_{\bx^{-1/2}}z}\be=(P_{\bx^{-1/2}}\bz)^2.$$
% From this we get: (CHECKwhere did we use $\bz\in \cq$?)
% $$P_{\bx^{-1/2}}\bz=(P_{\bx^{-1/2}}\by)^{1/2}$$ 
% and thus 
%  $$\bz=P_{\bx^{1/2}}(P_{\bx^{-1/2}}\by)^{1/2}.$$
%  \end{proof}

Finally, the geometric mean satisfies the following useful properties (e.g. see \cite{sturm}): 
\beq\label{meanwithe}
  \bx\#\by=\by\#\bx, \quad     (\bx\#\by)^{-1}=\by^{-1}\#\bx^{-1}, \quad  \bx\#\be=\bx^{1/2}.
  \eeq

%
%\noindent {\bf The Geometric Mean.}  Given a symmetric cone $\cc$, the {\em (metric) geometric mean} of two elements $\bx, \by\in \itr(\cc)$ is defined as the scaling point that takes $x^{-1}$ to $y$ \cite{sturm}:
%\beq\label{cmean2}
%P({\bx\#\by})\bx^{-1}=\by \quad \text{where} \quad \bx\#\by=P(\bx^{1/2})(P({\bx}^{-1/2})\by)^{1/2}.
%\eeq
%%It follows from \eqref{scalingp} that
%%\beq\label{cmean}
%%\bx\#\by=P(\bx^{1/2})(P({\bx}^{-1/2})\by)^{1/2}.
%%\eeq
%It can be further shown that the geometric mean is the unique solution $w \in {\rm int}(\cc)$ satisfying $P(w) x^{-1} = y$.  In addition, the geometric mean satisfies the following properties (e.g. see \cite{sturm}): 
%\beq\label{meanwithe}
%\begin{aligned}
%\bx\#\by & =\by\#\bx, \\
%(\bx\#\by)^{-1} & =\by^{-1}\#\bx^{-1}, \\
%\bx\#\be & =\bx^{1/2}.
%\end{aligned}
%\eeq

\begin{table}
\centering
\begin{tabular}{|c|c|c|}
\hline
&&\\
Cone of squares & PSD matrices & Second-order cone \\
& &\\
\hline
\hline
Ambient dimension & $X \in \mathbb{R}^{n \times n}$ & $(t,x) \in \mathbb{R} \times \mathbb{R}^{n}$ \\
\hline
$\circ$ & $X \circ Y = \frac{XY+YX}{2}$ & $(t,x) \circ (s,u) = (st + \la x,u\ra, t u + s x)$ \\
\hline
Identity & $I$ & $(1,0,\ldots,0)$ \\
\hline
Trace & $\sum X_{ii}$ & $2t$ \\
\hline
Rank (of interior) & $n$ & $2$ \\
\hline
Eigendecomposition & Spectral decomposition & $(t+\|x\|) \left(\begin{array}{c}
\frac{1}{2} \\ \frac{x}{2\|x\|} \end{array} \right)
+
(t-\|x\|) \left(\begin{array}{c}
\frac{1}{2} \\ -\frac{x}{2\|x\|} \end{array} \right)
$\\
\hline
$x^{1/2}$ & $X^{1/2}$ & $\sqrt{t+\|x\|} \left(\begin{array}{c}
\frac{1}{2} \\ \frac{x}{2\|x\|} \end{array} \right)
+
\sqrt{t-\|x\|} \left(\begin{array}{c}
\frac{1}{2} \\ -\frac{x}{2\|x\|} \end{array} \right)
$ \\
\hline
$x^{-1}$ & $X^{-1}$ & $\frac{1}{t+\|x\|} \left(\begin{array}{c}
\frac{1}{2} \\ \frac{x}{2\|x\|} \end{array} \right)
+
\frac{1}{t-\|x\|} \left(\begin{array}{c}
\frac{1}{2} \\ -\frac{x}{2\|x\|} \end{array} \right)
$ \\
\hline
\end{tabular}
\caption{Summary of key properties of the EJA associated to the cone of positive semidefinite matrices and the second order cone.}
\label{fig:summary_EJA}
\end{table}
\section{Deriving the Symmetric-cone  Multiplicative Update Algorithm}

Let $\cc\subseteq \R^d$ be a symmetric cone,  
% Consider  a symmetric cone . As described in Section~2, the symmetric cone $\cc$ can be expressed as   the cone of squares of a certain Euclidean Jordan Algebra $(\cj,\circ)$ with inner product $\la x,y\ra={\rm tr}(x\circ y)$.  
and let $X\in \R^{m\times n}_+$ be an entrywise non-negative matrix.  Our goal is to compute vectors  $ \ba_1,\ldots, \ba_{m} $ and $ \bb_1,\ldots, \bb_{n} $ belonging to the cone $ \cc $ such that
$$
X_{ij} \approx \langle \ba_i, \bb_j \rangle, \text{ for all } i\in [m], j\in [n].
$$
More formally, we frame our problem as a minimization instance over the squared loss objective:
%In order We to find a cone factorization over a symmetric cone $\cc$   of a given matrix $X \in \mathbb{R}_+^{m \times n}$ by
\begin{equation}\label{opt}
\underset{\ba_i\in \cc,\ \bb_j\in \cc}{\arg \min} \sum_{i,j} \left( X_{ij} - \langle \ba_i, \bb_j \rangle \right)^2.\qquad %\mathrm{s.t.} \qquad \ba_i, \bb_j \in \cc.
\end{equation}
In order to  (approximately) solve \eqref{opt}, we alternate between minimizing the $\ba_i$'s and the $\bb_j$'s.  Consider the sub-problem corresponding to fixing the variables $\{\ba_i\}$ and minimizing over $\{\bb_j\}$.  The objective \eqref{opt} is separable in the $\bb_j$'s.  After dropping the suffix $j$ to simplify our notation, the problem simplifies~to
 \begin{equation}\label{main}
\underset{\bb}{\arg \min} ~ \| \ca \bb - \bx \|_2^2 \qquad \mathrm{s.t.} \qquad \bb \in \mathcal{K},
\end{equation}
where $\bx$ is a fixed vector (the $j$-th column of $X$) and $\ca:\R^d\to \R^{m}$  is the linear mapping 
$$ \bx \mapsto  \ca \bx=( \la \ba_1,\bx \rangle, \ldots ,\la \ba_{m},\bx\ra)^\sfT.$$

We solve the convex optimization problem \eqref{main} via the Majorization-Minimization (MM) approach (see, e.g.,~\cite{Lange} and references therein).  In order to solve a problem  $\min \{f(x): x\in \mathcal{X}\}$ using the  MM framework, we need to  identify a family of \emph{auxilliary} functions $u_x: {\rm dom} f\to \R$, indexed by $x\in {\rm dom} f$, satisfying these conditions% the following two properties:
\beq\label{MM}
\begin{aligned}
f(y)&\le u_x(y),  \text{ for all } y\in \mathcal{X}, \text{ and} \\
f(x)&=u_x(x).
\end{aligned}
\eeq
%\item ${\rm argmin} \{ u_x: x\in {\rm dom}(f)\}$ is easy to find. 
The MM update scheme is given by
$$ x^{k+1} =\arg \min  \{ u_{x^{k}}(y):\ y \in \mathcal{X}\}.$$
Subsequently, the objective $f$ is non-increasing along the tracjectory of $x^k$, which follows easily from the inequalities
$$ f(x^{k+1})\le u_{x^{k}}(x^{k+1})\le u_{x^{k}}(x^{k})=f(x^{k}).$$
Recall that in our setting, the goal is to minimize  the function
$f(\bb) = \frac{1}{2} \|  \ca \bb - \bx \|_2^2$
over the symmetric cone $\cc$.  As the objective is quadratic, the Taylor expansion of $f$ at the past iterate $\bb_k$ is
\beq\label{taylor}
 f(\bb)=f({\bb}_k) +  (\bb-{\bb}_k)^\sfT   \nabla f({\bb}_k) + \frac{1}{2} \|\ca(\bb-{\bb}_k)\|_2^2.
\eeq
Suppose we restrict our search to quadratic auxilliary functions of the form
\beq\label{auxilliaryf}
u_{\bb_k}(\bb)=f({\bb}_k) + (\bb-{\bb}_k)^\sfT \nabla f({\bb}_k) + \frac{1}{2}  (\bb-\bb_k)^\sfT P(\bw)(\bb-\bb_k),
\eeq
for some appropriate $\bw\in {\rm int }(\cc)$. Using such an auxilliary function, to compute the next iterate, we need to minimize  the function in \eqref{auxilliaryf} over the cone $\cc$
\beq\label{auxproblem}
b_{k+1}=\underset{\bb\in \cc}{\arg \min}\ u_{\bb_k}(\bb).
\eeq
Note that although the objective \eqref{auxilliaryf} is strictly convex since choosing $ w\in {\rm int }(\cc)$ ensures that $P(w)$ is positive definite, it is not immediately clear how one handles the constraint $\bb\in \cc$.  Instead, we show that it is possible to make a specific choice of $w\in {\rm int}(\cc)$ such that the minimizer of the unconstrained auxilliary problem \eqref{auxproblem} lies in the interior of $\cc$.  In such a setting, we would have
\beq\label{gergre}
\underset{\bb\in \cc}{\arg \min}\ u_{\bb_k}(\bb)=\underset{\bb\in \R^d}{\arg \min}\ u_{\bb_k}(\bb),
\eeq
%Assuming that  \eqref{gergre} holds, the update $b_{k+1}$  can be calculated in closed-form as the minimizer of an  unconstrained strictly convex quadratic $u_{b_k}(b)$, and by setting the gradient equal to zero, we get:
and hence we can simply calculate $\bb_{k+1}$ as the unconstrained minimum of \eqref{auxilliaryf}, namely
\beq\label{ufsef}
\bb_{k+1}=\bb_k-P(\bw^{-1})\nabla f({\bb}_k)=\bb_k-P(\bw^{-1})(\ca^\sfT \ca\bb_k-\ca^\sfT \bx).
\eeq
Furthermore, suppose we pick $\bw\in {\rm int}(\cc)$ so that 
\beq\label{vdsfbsf}
\bb_k-P(\bw^{-1})\ca^\sfT \ca\bb_k=0.
\eeq
Then, the MM update $b_{k+1}$ is given by
$$\bb_{k+1}=P(\bw^{-1})\ca^\sfT \bx=\sum_{i=1}^{m}x_i P({\bw^{-1}})\ba_i.$$
Suppose that $\ba_i\in {\rm int}(\cc)$, and that the vector $x$ is non-zero.  From \eqref{invariance}, it follows that $\bb_{k+1}$ is a non-negative linear sum of elements in ${\rm int}(\cc)$, and hence $\bb_{k+1} \in {\rm int}(\cc)$ so long as $x\ne 0$.
%which immediately implies that  $\bb_{k+1} \in {\rm int}(\cc)$. Indeed,  we have that $\ba_i\in \cc$ and thus, $P({\bw^{-1}})\ba_i\in \cc$ (recall  \eqref{invariance}). Lastly, we have that $\sum_ix_i P({\bw^{-1}})\ba_i \in \cc$ as it is a conic combination of elements of $\cc$ (recall that $X$ is by assumption entrywise nonnegative, and $x=(x_i)_i$ is  its $i$-th column). 

The equation \eqref{vdsfbsf} specifies $\bw$ uniquely; specifically, \eqref{vdsfbsf} is equivalent to $P(\bw)\bb_k =\ca^\sfT \ca\bb_k$, which by the scaling point interpretation of the geometric mean (recall \eqref{cmean2}) implies that 
\beq\label{specialw}
\bw=\bb_k^{-1}\#\ca^\sfT \ca\bb_k.
\eeq
Note that for \eqref{specialw} to exist, we require $\bb_k$ and $\ca^\sfT \ca\bb_k$ to be in ${\rm int }(\cc)$. 

Finally, to check that the  function $u_{b_k}(b)$ given in   \eqref{auxilliaryf} corresponding to $\bw=\bb_k^{-1}\#\ca^\sfT \ca\bb_k
$
is indeed an auxilliary function, we need to verify that the two conditions \eqref{MM} are  satisfied. The proof of this fact is deferred to the next section and is the main technical result in this paper. 

  Summarizing the preceding discussion, employing the MM approach to minimize the function
$f(\bb) = \frac{1}{2} \|  \ca \bb - \bx \|_2^2$
over the symmetric cone $\cc$  using an auxilliary function of the form \eqref{auxilliaryf} with $\bw=\bb_k^{-1}\#\ca^\sfT \ca\bb_k
$
leads to the following  the update rule:
$$\bb_{k+1}=%P^{-1}(\bb_k^{-1}\#\ca^\sfT \ca\bb_k)\ca^\sfT \bx
P(\bb_k\#(\ca^\sfT \ca\bb_k)^{-1})\ca^\sfT \bx.$$%\text{ where }\  \bw=,$$
%which are exactly the update rule employed in Algorithm \ref{algo}. 
In a similar fashion, we get an  update rule for the $a_i$'s when the $b_j$'s are fixed, where $\ca$ is replaced~by
$$ \bx \mapsto  \mathcal{B} \bx=( \la \bb_1,\bx \rangle, \ldots ,\la \bb_{\mm},\bx\ra).$$
%In a similar fashion, our update rule for the $a_i$'s when the $b_j$'s are fixed is as follows
%$$ \bx \mapsto  \mathcal{B} \bx=( \la \bb_1,\bx \rangle, \ldots ,\la \bb_{n},\bx\ra)^\sfT.$$
We summarize our  procedure for computing factorizations over symmetric cones in Algorithm \ref{algo}.
 
\begin{algorithm}[h!]
\textbf{Input:} A non-negative matrix   $X\in \R^{m\times n}_{+}$\\
    \textbf{Output:}    Vectors $\ba_1,\ldots, \ba_{m}, \bb_1, \ldots, \bb_{n}\in \cc$ with  $X_{ij}\approx \la \ba_i, \bb_j\ra, \forall i,j$
 \begin{algorithmic}
  \STATE While stopping criterion not satisfied: 
  \STATE \begin{equation*}
\begin{aligned}
& a_i \leftarrow P(a_i\#(\mathcal{B}^\sfT \mathcal{B}a_i)^{-1})\mathcal{B}^\sfT X_{i :}, \text{ for all } 1\le i \le m \\
& \bb_j \leftarrow P(\bb_j\#(\ca^\sfT \ca\bb_j)^{-1})\ca^\sfT X_{:j}, \text{ for all } 1\le j\le n
\end{aligned}
%P_{\bw_j} ( A^{\sfT} \bx_j ) \qquad \text{ where } \qquad \bw_j := \bb_j^{-1}\#A^{\sfT} A \bb_j .
\end{equation*}  
\end{algorithmic}
\caption{Symmetric-Cone Multiplicative Update (SCMU) algorithm}
\label{algo}
\end{algorithm}
  
  \subsection{Two important special  cases}
%  
%  In the NMF problem we are given    $X\in \R_+^{m\times n}$    and a user-specified parameter $r\in~\mathbb{N}$, and the goal is to find matrices $A\in \R^{m\times r}_{+}$ and $B\in \R_{+}^{r\times n}$ satisfying $X=AB$. One of    the most widely used  approaches for calculating NMFs is the multiplicative update  rule introduced by Lee and Seung  in~\cite{LS00}. Specifically, using the Lee-Seung algorithm, the updates are given by
%\[\an=  \ao \circ {X\bo^\top \over \ao\bo\bo^\top} \quad \bn= \bo\circ {\an^\top X\over \an^\top \an \bo},\]
%where $X\circ Y, X/Y$ denote  the componentwise multiplication, division of two matrices respectively. 
%Letting  $\bn=[\sbn_1\ldots \sbn_n]$ it will be useful to equivalently express  the Lee-Seung updates with respect to the columns of $\bn$ (and similarly for $\an$), i.e.,  
%%$$\sbn_j(k)=\sbo_j(k) {\an^\sfT X_{:j} \over \an^\top \an\sbo_j }.$$
%\be\label{vdbdfg}
%\sbn_j(k)=\sbo_j(k) {\sum_{i=1}^m\san_i(k)X_{ij} \over \sum_{i=1}^m\san_i(k)\la \san_i,\sbo_j\ra}.
%\ee

First, we specialize the SCMU algorithm to the setting where $\cc=\R^k_+$, in which case the SCMU algorithm  gives an iterative method for computing NMFs. The non-negative orthant is the cone of squares of the EJA $(\R^k, \circ)$, where  the Jordan product is componentwise multiplication, i.e., $x\circ y={\rm diag}(x)y$. Moreover, the trace of $x$ is just its  1-norm, the Lyapunov transformation is $L(x)={\rm diag}(x)$ and  the quadartic mapping $P(x)={\rm diag}(x)^2$. Finally, the metric geometric mean~\eqref{cmean} of $x, y\in \R^k_{++}$ is given by $x\#y=(\sqrt{x_1y_1}, \ldots, \sqrt{x_ky_k})$. 

Putting everything together, the $\ell$-th coordinate of the vector $b_j$ is updated by:
 %In the NMF problem we are given    $X\in \R_+^{m\times n}$    and a user-specified parameter $r\in~\mathbb{N}$, and the goal is to find matrices $A\in \R^{m\times r}_{+}$ and $B\in \R_{+}^{r\times n}$ satisfying $X=AB$. One of    the most widely used  approaches for calculating NMFs is the multiplicative update  rule introduced by Lee and Seung  in~\cite{LS00}. Specifically, using the Lee-Seung algorithm, the updates are given by
%\[\an=  \ao \circ {X\bo^\top \over \ao\bo\bo^\top} \quad \bn= \bo\circ {\an^\top X\over \an^\top \an \bo},\]
%where $X\circ Y, X/Y$ denote  the componentwise multiplication, division of two matrices respectively. 
%Letting  $\bn=[\sbn_1\ldots \sbn_n]$ it will be useful to equivalently express  the Lee-Seung updates with respect to the columns of $\bn$ (and similarly for $\an$), i.e.,  
%%$$\sbn_j(k)=\sbo_j(k) {\an^\sfT X_{:j} \over \an^\top \an\sbo_j }.$$
\beq
b_j(\ell)\leftarrow b_j(\ell) {\sum_{i=1}^ma_i(\ell)X_{ij} \over \sum_{i=1}^ma_i(\ell)\la a_i,b_j\ra}.
\eeq
These updates correspond to the multiplicative update  rule introduced by Lee and Seung  in~\cite{LS00}, which is one
    of the most widely used  approaches for calculating NMFs.
    
  Next, we specialize  the SCMU algorithm for cone  factorizations with $k\times k$  PSD factors (with real entries). Letting $\mathbb{S}^k$ denote the space of $k\times  k$ real symmetric matrices, the (real) $k\times k$ PSD cone   is the cone of squares of the EJA $(\mathbb{S}^k,\circ)$, where  the Jordan product is given by $X\circ Y=(XY+YX)/2$. In this setting, the Lyapunov operator $L(X)$ and the quadartic representation $P(X)$ are superoperators (i.e.,  linear operators acting on a vector space of linear operators), and are concretely  given by 
  $$
  \begin{aligned}
  {\rm vec}(L(X)(Y))&={1\over 2}((X\otimes I)+(I\otimes X)){\rm vec}(Y)\\
   {\rm vec}(P(X)Y)&=(X\otimes X){\rm vec}(Y),
   \end{aligned}
  $$
  where ${\rm vec}(\cdot)$ is the vectorization operator. Using  that ${\rm vec}(ABC)=(C^\top \otimes A){\rm vec}(B)$, we~get
  \beq\label{csdvsdv}
  P(X)Y=XYX.
  \eeq
  Moreover, the trace is just the usual trace of a symmetric matrix (i.e., the sum of its eigenvalues). Lastly, using \eqref{csdvsdv}, the metric geometric mean \eqref{cmean} specializes to  	
  $$X\# Y= X^{1/2} (X^{-1/2} Y X^{-1/2})^{1/2} X^{1/2},$$
  the usual geometric mean of two positive definite matrices, e.g. see \cite{bhatia}. 
  Putting everything together,  we have that
\begin{equation*}
\begin{aligned}
\mathcal{A} : \mathbb{S}^{k} \rightarrow \mathbb{R}^{m} & \quad  Z\mapsto \left(
\mathrm{tr}(A_1 Z), \ \ldots \ 
,\mathrm{tr}(A_m Z) \right)^\top\\
\mathcal{A}^\top:  \R^m\rightarrow \mathbb{S}^k & \quad x \mapsto \sum_{i=1}^mx_iA_i
\end{aligned}
\end{equation*}
and  thus, the SCMU algorithm in the case $\cc=\mathbb{S}^k_+$ specializes to:
$$B_j \leftarrow S_j(\ca^\sfT X_{:j})S_j,  \text{ where } S_j=B_j\#([\ca^\sfT \ca]B_j)^{-1}.$$
This is exactly the Matrix Multiplicative update method derived in \cite{YSV}.

\section{Performance Guarantees of the SCMU Algorithm}

Our first result is to show that, given a symmetric cone $\cc\subseteq \R^d$, the function $u_{\tb}(\bb)$ defined in \eqref{auxilliaryf} does indeed parameterize a family of auxilliary functions for $f(\bb) = \frac{1}{2} \|  \ca \bb - \bx \|_2^2$ over $\cc$.  To do so, we need to verify that the two properties given in \eqref{MM} do indeed hold.  We obviously have that $f(\tb)=u_{\tb}(\tb)$, and hence it remains to verify  the domination property, namely  that  
\beq
f(\bb)\le u_{\tb}(\bb), \ \text{ for all }  \bb \in \cc.
\eeq 
In fact, we show in the next theorem that this bound holds for all $b\in \R^d$. %a slightly stronger result. 
%, namely 
%\beq\label{dom}
%f(\bb)\le u_{\tb}(\bb), \ \text{ for all }  \bb \in \R^d.
%\eeq 

\begin{theorem}\label{auxilliary}
Let $\bb_k \in{\rm int}(\mathcal{K})$ and define  $\ca$ to be the linear map
$$ \bx \mapsto  \ca \bx=( \la \ba_1,\bx \rangle, \ldots ,\la \ba_{\mm},\bx\ra)^\sfT,$$
where $a_1,\ldots, a_{\mm}\in {\rm int}(\cc)$.
 %Setting   $\bw=\tb^{-1}\#A^\sfT A\tb$ we have that, 
Furthermore, let  $f(\bb) = \frac{1}{2} \|  \ca \bb - \bx \|_2^2$ and 
$$u_{\tb}(\bb)=f(\tb) + (\bb-\tb)^\sfT  \nabla f(\tb) + \frac{1}{2}   (\bb-\tb)^\sfT P(\tb^{-1}\#\ca^\sfT \ca\tb)(\bb-\tb).$$
Then, we have that 
\begin{equation}\label{maintechinal}
\ca^\sfT \ca \preceq  P({\tb^{-1}\#\ca^\sfT \ca \tb}).
\end{equation}
In particular, this immediately implies that 
\beq\label{dom}
f(b)\le u_{\tb}(\bb),  \text{ for all }\bb \in \mathbb{R}^{d}.
\eeq
\end{theorem} 

%Replacing $f$ with its Taylor expansion at  $\tb$ (recall \eqref{taylor}), equation \eqref{dom}  boils down to:
%$$ (\bb-\tb)^\sfT \ca^\sfT \ca(b-b_k)\le  (\bb-\tb)^\sfT  P(\tb^{-1}\#\ca^\sfT \ca\tb)(\bb-\tb), \text{ for all }  \bb \in \R^d.$$
%Consequently, to verify the domination property it suffices to show that the following operator is positive semidefinite
%\begin{equation*}
%P({\tb^{-1}\#\ca^\sfT \ca \tb})-\ca^\sfT \ca \succeq 0.
%\end{equation*}
%Our first step is to reduce to the case where $b=e$.

\begin{proof} 
%Given  a general EJA, we can decompose it as a direct sum of simple ones.   The operators $\ca^\sfT \ca$ and $P(\ca^\sfT \ca e)^{1/2}$ are separable with respect to  the blocks corresponding to these individual simple EJAs.  Thus, if the operator inequality $\ca^\sfT \ca \preceq P(\ca^\sfT \ca e)^{1/2}$ holds for each individual block, it holds for the full-sized operators. 

We first focus on simple EJAs. We need to show the validity of the generalized inequality~\eqref{maintechinal} with respect to the Euclidean inner product. Nevertheless, as for simple EJAs, any inner product is a positive multiple of the canonical one, it suffices to prove  \eqref{maintechinal} for the canonical one. 

The proof  of \eqref{maintechinal} is broken down in two steps. 
First, we show that it suffices to consider the special case where $\bb_k=\be$.  More precisely, we first prove the inequality
\beq\label{dominatione}
\tilde{\ca}^\sfT \tilde{\ca} \preceq   P(e\#\tilde{\ca}^\sfT \tilde{\ca} e)
\eeq
for all $\ca$ such that $\tilde{a_i} \in \mathrm{int}(\cc)$.  Assuming that \eqref{dominatione} holds, we let $\tilde{\ca} = \ca P(b^{1/2})$, and let $\tilde{\bw} = (\tilde{\ca}^\sfT \tilde{\ca} \be)^{1/2}$.  Then
\begin{equation*}
\tilde{\bw} = (\tilde{\ca}^\sfT \tilde{\ca} \be)^{1/2} = (P(b^{1/2}) \ca^\sfT \ca P(b^{1/2}) \be)^{1/2} = (P(b^{1/2}) \ca^\sfT \ca b)^{1/2}.
\end{equation*}
Thus, by \eqref{cmean} we have %and \eqref{cmean2}, we have
\begin{equation*}
P(b^{-1/2}) \tilde{\bw} = P(b^{-1/2}) (P(b^{1/2}) \ca^\sfT \ca b)^{1/2} = b^{-1} \# \ca^\sfT \ca b.
\end{equation*}
By \eqref{dominatione}   we have $\tilde{\ca}^\sfT \tilde{\ca} \preceq P(\tilde{\ca}^\sfT \tilde{\ca} e)^{1/2} = P(\tilde{\bw})$.  
Consequently, we have 
\begin{equation*}
P(\bb^{-1}\#\ca^\sfT \ca\bb) =P(P(\bb^{-1/2})\tilde{\bw}) =  P(\bb^{-1/2}) P(\tilde{\bw})  P(\bb^{-1/2}) \succeq P(\bb^{-1/2}) \tilde{\ca}^{\sfT} \tilde{\ca} P(\bb^{-1/2}) = \ca^{\sfT}\ca,
\end{equation*}
where in the second equality we used  \eqref{unfolding}.
To conclude the proof, we prove  \eqref{dominatione} by showing the following two properties in   Lemma  \ref{thm:squareroot_ineq} and Lemma \ref{svdfs} respectively:
\begin{align}
  %\sum_iP_{\ba_i}^{1/2} \preceq P_{({\sum_i \ba_i}^{1/2})}
\la b, P^{1\over 2}({{a_1}})b\ra+ \la b, P^{1\over 2}(a_2)b\ra & \le \la b, P^{1\over 2}({{(a_1+a_2}}))b\ra,\  \text{ for all }  \bb \in \R^d. \label{sumineq}\\
 \la a, b\ra^2 & \le   {\rm tr}(\ba)\la b, P^{1\over 2}({{a}})b\ra,\ \text{ for all }  \bb \in \R^d.\label{traceineq}
 \end{align}

%The proof of Step 1 is  given in Lemma \ref{step1}, and the proof of Step 2  follows  (where we respectively prove \eqref{sumineq} and \eqref{traceineq}).

Assuming the validity of \eqref{sumineq} and \eqref{traceineq},  it is now easy   to conclude    that~\eqref{dominatione} holds. By definition of the map $\ca$ and  \eqref{meanwithe}  we~have 
$$\be\#(\ca^\sfT \ca\be)=(\ca^\sfT \ca\be)^{1/2}=\left(\sum_{i=1}^{m}\la \ba_i , \be\ra\ba_i\right)^{{1\over 2}}=\left(\sum_{i=1}^{m}{\rm tr}(\ba_i \circ \be)\ba_i\right)^{{1\over 2}}=\left(\sum_{i=1}^{m}{\rm tr}(\ba_i)\ba_i\right)^{{1\over 2}}.$$
Consequently,   \eqref{dominatione} is equivalent to
\beq\label{termbyterm}
\sum_{i=1}^{m}\la \ba_i, b\ra^2\le  \left\la b, P^{1\over 2}\left({ \sum_{i=1}^{}{\rm tr}(\ba_i)\ba_i}\right) b\right\ra \eeq
Lastly, the proof of \eqref{termbyterm} follows from the following chain of inequalities:
\begin{eqnarray*}
\sum_{i=1}^{m}\la \ba_i, b\ra^2 & \le & \sum_{i=1}^{m}{\rm tr}(\ba_i)\la b, P({{a_i}}^{1/2})b\ra\\
& = & \sum_{i=1}^{m}\la b, P^{1/2}({{{{\rm tr}(\ba_i)\ba_i}}})b\ra \\
% &\sum_{i=1}^{m}\la b, \sqrt{P}_{{\rm tr}(\ba_i)\ba_i}b\ra\le \\
& \le & \left\la b, P^{1 \over 2}({ \sum_{i=1}^{}{\rm tr}(\ba_i)\ba_i}) b\right\ra,
 \end{eqnarray*}
where for the first inequality we use \eqref{traceineq}, for the second equality we use property \eqref{scalar} of the quadratic representation, and for the last inequality we use \eqref{sumineq}. 

Lastly, we consider the case where the EJA  is a direct sum of simple ones. In this case, the cone of squares $\cc$ is a direct of simple symmetric cones, i.e., $\cc=\cc_1\oplus  \ldots \oplus \cc_k$ and the Jordan product is given by 
$(x_1,\ldots, x_k)\circ (y_1,\ldots,y_k)=(x_1\circ y_1, \ldots, x_k\circ y_k)$ and ${\rm tr}(x_1,\ldots, x_k)=\sum_{i=1}^k{\rm tr}(x_i).$ Then, it is easy to check that the operators $\ca^\sfT \ca$ and $P(\ca^\sfT \ca e)^{1/2}$ are separable with respect to  the blocks corresponding to these individual simple EJAs.  Thus, if the operator inequality $\ca^\sfT \ca \preceq P(\ca^\sfT \ca e)^{1/2}$ holds for each individual block, it holds for the full-sized operators. 
\end{proof}
%We now proceed with the proof of Step 1. 
%
%
%\begin{lemma}\label{step1} Consider a  linear map 
%$ \bx \mapsto  \ca \bx=( \la \ba_1,\bx \rangle, \ldots ,\la \ba_{\mm},\bx\ra)^\sfT,$ where $a_i\in {\rm int}(\cc)$. Then, the generalized inequality   $  \ca^\sfT \ca \preceq   P(\ca^\sfT \ca e)^{1/2}
%$ implies that  
%$$\ca^\sfT \ca \preceq  P({b^{-1}\#\ca^\sfT \ca b}), \text{ for any } b\in {\rm int}(\cc).$$
%
%\end{lemma}
%\begin{proof}%[Proof of Theorem \ref{auxilliary}]
%\end{proof}
%

Next we proceed  with the proof of  \eqref{sumineq}.
\begin{lemma}\label{thm:squareroot_ineq}
For any  $\ba_1,\ba_2 \in \mathrm{int}(\cc)$ we have that
$$P^{1\over 2}({{a_1}}) + P^{1\over 2}(a_2) \preceq P^{1\over 2}({{a_1+a_2}}).$$
\end{lemma}
\begin{proof}We will show that
\begin{equation*}
\la b, P^{1\over 2}({{a_1}})b\ra+ \la b, P^{1\over 2}(a_2)b\ra  \le \la b, P^{1\over 2}({{a_1+a_2}})b\ra,\  \text{ for all }  \bb \in \R^d;
\end{equation*}
 For any EJA   $(\cj,\circ)$ the function
$$f: \cq\times \cq \to \R,  \quad (a,b) \mapsto {\rm tr}(P(k)a^p\circ b^{1-p}),$$
is jointly concave for any fixed $k\in \cj$ and $0\le p\le 1$   \cite[Theorem 3.1]{Fay}. This the extension of Lieb's Concavity Theorem in the more general setting of  EJAs. Also, by  \cite[Lemma 3.1]{Fay} we have
$${\rm tr}(P(k)a^p\circ b^{1-p})=\la k, P({a^p,b^{1-p}})k\ra,$$
 and thus, it follows that for any fixed $k\in \cj,$ the mapping 
$$(a,b) \mapsto\la k, P({{a}^{1/2}, {b}}^{1/2})k\ra,$$
is concave; i.e., for any $\lam\in [0,1]$ and $(a_1,b_1),(a_2,b_2)\in \cq\times \cq $ we have 
\begin{eqnarray*}
& & \la k, P(({{\lam a_1+(1-\lam)a_2})^{1/2}, ({\lam b_1+(1-\lam)b_2}})^{1/2})k\ra \\
& \ge & \lam  \la k, P({a_1}^{1/2}, {b_1}^{1/2})k\ra +(1-\lam) \la k,  P({a_2}^{1/2}, {b_2}^{1/2})k\ra.
\end{eqnarray*}
 Setting $a_1=b_1={a/\lam}$ and $a_2=b_2={b/(1-\lam})$, and using that $P({x,x})=P(x)$
 we get 
  $$ \la k, P^{1\over 2}({{ a+b}})k\ra \ge  \la k,   P^{1\over 2}({{a}})k\ra +\la k,   P^{1\over 2}({{b}})k\ra,$$
  which is exactly \eqref{sumineq}. 
\end{proof}

%\begin{lemma} \label{step1}
%Let $\ca$ be the linear map $ \bx \mapsto  \ca \bx=( \la \ba_1,\bx \rangle, \ldots ,\la \ba_{\mm},\bx\ra)^\sfT$, where $a_1,\ldots, a_{\mm}\in \cc$.  Then
%\begin{equation} \label{eq:maintechnical_simplified}
%\ca^\sfT \ca \preceq P(\ca^\sfT \ca e)^{1/2}.
%\end{equation}
%\end{lemma}

%
%We begin by proving Proposition \ref{thm:maintechnical_simplified} under the additional assumption that the EJA is \emph{simple}, in which we can interchangeably replace the regular inner product with the inner product $\lj \cdot,\cdot \rj$ associated with the EJA.  Under this assumption, we break down Proposition \ref{thm:maintechnical_simplified} into two simpler results, which we state and prove in the following.
%

Next we proceed  with the proof of \eqref{traceineq}.

\begin{lemma}\label{svdfs}
For any $a\in {\rm int}(\cc)$ we have that 
$$ \la a, b\ra^2  \le   {\rm tr}(\ba)\la b, P({{a^{1/2}}})b\ra,\ \text{ for all }  \bb \in \R^d.$$

\end{lemma}
\begin{proof}

 For this, note that:
  $$\langle \ba,\bb \rangle = \mathrm{tr}(\ba \circ \bb) = \mathrm{tr}(P({\ba^{1/4}}) \ba^{1/2} \circ \bb) = \mathrm{tr}(\ba^{1/2} \circ P({\ba^{1/4}}) \bb),$$
  where   for the last equality we used \eqref{Pxsymmetric}. Using that ${\rm tr}(a\circ b)^2\le {\rm tr}(a^2){\rm tr}(b^2)$ (see \cite{gowda}),  we~get
\beq\label{csvsdbvs}
\langle \ba,\bb \rangle^2={\rm tr}(a\circ b)^2\le {\rm tr}(a){\rm tr}(P({a^{1/4}})b \circ P({a^{1/4}})b).
\eeq
Lastly, we have that 
\beq\label{sdvsbv}
{\rm tr}(P({a^{1/4}})b \circ P({a^{1/4}})b)=\la P({a^{1/4}})b, P({a^{1/4}})b\ra=\la  b,P^2({a^{1/4}})b\ra=\la  b,P({{a}}^{1/2})b\ra,
\eeq
where  for the last equality we use \eqref{powers}. 
Combining \eqref{csvsdbvs} with~\eqref{sdvsbv} we get $\langle \ba,\bb \rangle^2\le{\rm tr}(a) \la  b,P({{a}}^{1/2})b\ra.$ 
%This concludes the proof of Step 2, and consequently, Lemma \ref{thm:domination} has been established. 
\end{proof}

In our last result in this section  we show that fixed points of our update scheme correspond to first-order stationary points of the optimization problem \eqref{opt}.  

\begin{theorem}\label{thm:kkt} 
Let  $\{a_i\}_{i\in [\mm]}, \{b_i\}_{j\in [\nn]} \in {\rm int}(\cc)$  be fixed points of the  multiplicative update rule:
$$
\begin{aligned}
& a_i \leftarrow P(a_i\#(\mathcal{B}^\sfT \mathcal{B}a_i)^{-1})\mathcal{B}^\sfT X_{i :}, \text{ for all } 1\le i \le \mm \\
& \bb_j \leftarrow P(\bb_j\#(\ca^\sfT \ca\bb_j)^{-1})\ca^\sfT X_{:j}, \text{ for all } 1\le j\le \nn.
\end{aligned}
$$
Then $\{a_i\}_{i\in [\mm]}, \{b_i\}_{j\in [\nn]}$ satisfy %the KKT conditions  corresponding to: 
%$$\underset{\ba_i\in \cc,\ \bb_j\in \cc}{\arg \min} \sum_{i,j} \left( X_{ij} - \langle \ba_i, \bb_j \rangle \right)^2, %\mathrm{s.t.} \qquad \ba_i, \bb_j \in \cc.
%$$
%namely, they satisfy:
\beq\label{KKT}
\mathcal{B}^\sfT (X_{: i})=[\mathcal{B}^\sfT\mathcal{B}] (a_i), \ i\in [\mm] \quad \text{ and } \quad
 \mathcal{A}^\sfT (X_{: j}) =[\mathcal{A}^\sfT \mathcal{A}] (b_j), \ j \in [\nn].
\eeq

\end{theorem}

\begin{proof}We only focus on the $b_j$'s as the argument for the $a_i$'s is similar.  Assume that 
$$ \bb_j = P(w)\ca^\sfT X_{:j},\  \text{ where } w=\bb_j\#(\ca^\sfT \ca\bb_j)^{-1} \text{ for all } 1\le j\le \nn.$$
Since $a_i,\bb_j\in {\rm int}(\cc)$, we have $\la a_i,b_j\ra>0$ and hence $\ca^\sfT Ab_j=\sum_i\la a_i,b_j\ra a_i \in {\rm int} (\cc)$. %(recall the characterization of ${\rm int}(\cc)$ given in  \eqref{interiorconeofsquares}).  
Thus $w$ is invertible~and 
\beq\label{cxsdvdf}
P(w^{-1})\bb_j=\ca^\sfT X_{:j}.
\eeq
By noting $ w=\bb_j\#(\ca^\sfT \ca\bb_j)^{-1}$ and \eqref{meanwithe}, we have 
$$w^{-1}=(\bb_j)^{-1}\#(\ca^\sfT \ca\bb_j).$$
It follows from \eqref{cmean2} that
\beq\label{sdcdfv}
P(w^{-1})b_j=\ca^\sfT \ca\bb_j.
\eeq
By combining \eqref{cxsdvdf} and \eqref{sdcdfv}, we have $\ca^\sfT \ca\bb_j=\ca^\sfT X_{:j}$ for all  $ j \in [\nn]$.
\end{proof}

\section{Numerical Experiments}

In this section we proceed from theory to practise and use the SCMU algorithm for computing SOCP-lifts of regular $n$-gons. To the best of our knowledge there are no algorithms developed specifically  for computing SOCP-lifts. In terms of negative results,  Fawzi established in \cite{hamza} that the   $3\times 3$  positive semidefinite cone does not admit any second-order cone representation.

In the case of  the second-order cone, there are two different types  of lifts that can be considered. 
The first possibility is  to  consider lifts over $\soc_n=\{(x,t)\in \mathbb{R}^{n}\times \R: \|x\|_2\le t\} $, whereas the second possibility is to consider some fixed~$n$, say $n=2$, and consider lifts over Cartesian products of~$\soc_2$. Nevertheless, there is a close relationship between these two types of SOC-lifts. Specifically, in view of the ``tower of variables'' construction given  in \cite{BT}, a lift over $\soc_k$ (where $k=2^\theta$) can be transformed into a lift over the symmetric cone 
$$\underbrace{\soc_2\times \ldots \times \soc_2}_{k-1 \text{ times}},$$
by adding $k-2$ additional variables. As an example of this,  $(x_1,x_2,x_3,x_4,t)\in \soc_4$ iff
$$ \exists y_1,y_2 \text{ where } (x_1,x_2,y_1)\in \soc_2, (x_3,x_4,y_2)\in \soc_2, (y_1,y_2,t)\in \soc_2.$$
%
% There is also an  interesting  link between SOC and PSD-lifts. Indeed, it is well-known that  there exists an isometric bijection between 
% $\soc_2$ and $\mathbb{S}^2_+$, given by
%%We show that ${\rm rank}^2_{\soc}(M)$ coincides with ${\rm rank}_{psd}^2(M)$, the latter corresponding to matrix  factorizations using block diagonal psd matrices, where the  blocks have  size $2$. For this consider the  map:
%$$   (x_1,x_2,t) \mapsto  {1\over \sqrt{2}}\begin{pmatrix}t-x_2 & x_1\\x_1& t+y_2 \end{pmatrix}.
%  $$
%  More generally, 
%  there exists a linear isometry mapping $\soc_n$ to Hermitian PSD matrices of size~${2^{\lfloor \frac{n}{2} \rfloor}}$ in terms of the so-called  {Brauer-Weyl} matrices, e.g. see \cite{cpsd}. For a concrete example, when 
%  $n=2$ this mapping is given by
%$$
%(x_1, x_2, t)\mapsto {1\over \sqrt{2}}(tI_2+x_1X+x_2Y)={1\over \sqrt{2}}\begin{pmatrix}
%t & x_1-ix_2\\x_1+ix_2 & t
%\end{pmatrix},
%$$ 
%whereas for $n=4$ it is given by
%$$
%(x_1,x_
%2,x_3,x_4, t) \mapsto {1\over \sqrt{4}}(tI_4+x_1X\otimes I_2 +x_2Z\otimes X+x_3Y\otimes I_2+x_4Z\otimes Y)
%,$$
%where 
%$$X=\begin{pmatrix} 0 & 1\\1 & 0\end{pmatrix} ,Y=\begin{pmatrix}0 & -i\\i & 0\end{pmatrix}, Z=\begin{pmatrix} 1 & 0\\0 & -1\end{pmatrix},$$ are the Pauli matrices. As these mappings are isometric, every $\soc_n$ factorization gives rise to a PSD factorization (with Hermitian matrices) of size ${2^{\lfloor \frac{n}{2} \rfloor}}$.
%
%%Lastly, we have that $(x,t)\in \soc_n$ iff the corresponding arrow matrix 
%%$$\begin{pmatrix} t & x^\sfT\\x & tI\end{pmatrix}$$ is PSD. 
%

\subsection{Implementation details}

\paragraph{Damped updates.}  Recall that the update rule of the SCMU algorithm  is:
$$\bb \leftarrow  P(\bb \#(\ca^\sfT \ca\bb )^{-1})\ca^\sfT \bx.$$
In performing the update, it is necessary to compute square-roots and inverses of certain elements.  The conditioning of these steps depend on how close the eigenvalues of these elements are to zero.  As such, it is advisable to apply a small amount of damping when performing these steps.  We summarize these steps in Algorithm \ref{algo-damped} -- here, $e$ is the EJA identity element, while $J$ is the identity map.
\begin{algorithm*}[h]
\textbf{Input:} EJA element $a$, linear map $\mathcal{B}$, damping parameter $\epsilon$ 
\begin{algorithmic}
\STATE 1. $z \leftarrow (\mathcal{B}^\sfT \mathcal{B} a + \epsilon e)^{1/2} $
\STATE 2. $T \leftarrow P(z) + \epsilon I $
\STATE 3. $h \leftarrow T^{1/2}  (T^{-1/2} a+\epsilon e)^{1/2}$
\STATE 4. $b \leftarrow P(h) \mathcal{A}^\sfT x$
\end{algorithmic}
\label{algo-damped}\caption{Damped multiplicative updates for cone factorizations}
\end{algorithm*}
In our numerical experiments, we apply $\epsilon = 10^{-6}$ as our choice of damping.

\paragraph{Strategic initializations.}  We apply the following two stage initialization strategy.  In the first stage, we apply our algorithm for $100$ iterations over $100$ different random initializations.  We keep track of the final iterate and the residual error corresponding to all initializations.  We then eliminate all but $10$ iterates with smallest $10$ residual errors.  In the second stage, we apply our algorithm for an additional $900$ iterations starting from the final iterate of these $10$ iterates.  We report the smallest residual error obtained in the second stage.

\subsection{SOCP-lifts for regular $n$-gons}
Denoting by   $\times^l \soc_k$  the Cartesian product of $l$ copies of  where $\soc_k=\{(x,t)\in \mathbb{R}^{k}\times \R: \|x\|_2\le t\} $, in our experiments we compute $\times^l \soc_k$ 
 factorizations of the slack matrices of  regular $4$-gons, $5$-gons, $6$-gons, and $8$-gons for various values of $l$ and $k$. Subsequently, based on our   numerical results, we formulate several conjectures on the (non)existence of certain types of SOCP-lifts.

 Our guiding principle is a simple heuristic that was explicitly formulated and used in  \cite{VGG18} in the setting  of PSD factorizations. Specifically,   for a fixed value of  $k$, we would like to find the least~$l$ for which our matrix has a $ \times^l \soc_k$ factorization. Denoting by $l^*$ the least possible value such a lift exists, we expect to see a noticeable  ``phase transition''  with respect  to the error of the SCMU algorithm, namely it should be positive for $l<l^*$ and zero for $l\ge l^*$.

  We note that for a fixed value of $k$, there always exists a $\times^l \soc_k$-lift for a large enough value of~$l$. Indeed, for any non-negative matrix $X\in \R^{m\times n}_+$ we always have a non-negative factorization with vectors  of dimension   $\min \{n,m\}$. Considering any NMF with an even dimension $d$, we can pair the non-negative coordinates in pairs of two to get a factorization $\underbrace{\R^2_+\times \ldots \times \R^2_+}_{d/2\text{ times}}$. Lastly, as $\R^2_+$ can be rotated to $\soc_1$, this construction leads to a $\times^{d/2} \soc_k$ factorization.

In practise, for fixed $k$ and increasing $l$, we notice that the errors steadily decrease up to a point where the error stagnates.  
Based on this, we conjecture that the first instance where the error stabilizes corresponds to smallest $l$ for which there is an exact $\times^l \soc_k$-lift.

\if0
This is the slack matrix
\begin{equation*}
\sqrt{2} \left( \begin{array}{cccc}
0 & 1 & 1 & 0 \\ 
0 & 0 & 1 & 1 \\
1 & 0 & 0 & 1 \\
1 & 1 & 0 & 0 \\
\end{array} \right)
\end{equation*}
\fi
In the first instance, we compute  factorizations of the slack matrix of the regular $4$-gon.  In Figure \ref{fig:numexp4gon}, we show the final residual errors obtained using our method. % In computing the residual errors corresponding to a $\times^l \soc_k$ factorization, we observed the following trends.  
\begin{figure}[h]
\centering
\begin{tabular}{| c || c c c |}
\hline
 & $\times^1$ & $\times^2$ & $\times^3$ \\
\hline
\hline
$\soc_1$ & 0.50 & 0.0019 & 0.0025 \\
$\soc_2$ & 0.17 & 0.0020 & 0.0025\\
$\soc_3$ & 0.17 & 0.0021 & 0.0027 \\
$\soc_4$ & 0.17 & 0.0021 & 0.0027 \\
\hline
\end{tabular}
\caption{Best error for regular $4$-gon.}
\label{fig:numexp4gon}
\end{figure}
For $l=1$ and increasing $k$, we noticed that the errors decrease from $k=1$ to $k=2$, but stagnate right after.  This suggests  that the $4$-gon does not admit a $\soc_k$-lift for any $k$.
  In the case of the $4$-gon, we have an explicit factorization of the slack matrix using $\soc_1 \times \soc_1$:
\begin{equation*}
\sqrt{2} \left( \begin{array}{rrrr}
0 & 1 & 1 & 0 \\ 
0 & 0 & 1 & 1 \\
1 & 0 & 0 & 1 \\
1 & 1 & 0 & 0 \\
\end{array} \right)
= \frac{\sqrt{2}}{12} \left( \begin{array}{rrrr}
3 & -3 & & \\  
3 & 3 & & \\
 & & 2 & 2 \\
 & & 2 & -2 \\
\end{array} \right)
\left( \begin{array}{rrrr}
 & 2 & 4 & 2 \\
 & -2 & & 2 \\
6 & 3 & & 3 \\ 
 & -3 & & 3 \\
\end{array} \right).
\end{equation*}

%This is the optimal approximation for a single copy of $\mathrm{SOCP}(1)$:
%\begin{equation*}
%\frac{\sqrt{2}}{2} \left( \begin{array}{rrrr}
%1 & 2 & 1 & 0 \\ 
%1 & 0 & 1 & 2 \\
%1 & 0 & 1 & 2 \\
%1 & 2 & 1 & 0 \\
%\end{array} \right)
%= \frac{\sqrt{2}}{4}
%\left( \begin{array}{rr}
%1 & -1 \\ 
%1 & 1 \\
%1 & 1 \\
%1 & -1 \\
%\end{array} \right)
%\left( \begin{array}{rrrr}
%1 & 1 & 1 & 1 \\ 
%0 & -1 & 0 & 1 \\
%\end{array} \right).
%\end{equation*}

\paragraph{$5$-gon.} In  our second example, we compute a factorization of the slack matrix of the regular $5$-gon.  In Figure \ref{fig:numexp5gon}, we show the final residual errors obtained using the SCMU  method and  based on these, conjecture that the regular $5$-gon admits a $\times^l \soc_k$-lift if and only if $l \geq 3$.

\begin{figure}[h]
\centering
\begin{tabular}{| c || c c c c |}
\hline
 & $\times^1$ & $\times^2$ & $\times^3$ & $\times^4$ \\
\hline
\hline
$\soc_1$ & 0.47 & 0.12 & 0.0024 & 0.0026\\
$\soc_2$ & 0.10 & 0.018 & 0.0026 & 0.0027 \\
$\soc_3$ & 0.10 & 0.018 & 0.0034 & 0.0033 \\
$\soc_4$ & 0.10 & 0.018 & 0.0040 & 0.0035 \\
\hline
\end{tabular}
\caption{Best error for regular $5$-gon}
\label{fig:numexp5gon}
\end{figure}

\paragraph{$6$-gon.}  In the third instance, we compute factorization of the slack matrix of the regular $6$-gon.  In Figure \ref{fig:numexp6gon}, we show the final residual errors obtained using the SCMU  method  and conjecture that the regular $5$-gon admits a $\times^l \soc_k$-lift if and only if $l \geq 3$.

\if0
This is the slack matrix
\begin{equation*}
\frac{\sqrt{3}}{2} \left( \begin{array}{rrrrrr}
0&1&2&2&1&0\\ 
0&0&1&2&2&1\\ 
1&0&0&1&2&2\\ 
2&1&0&0&1&2\\ 
2&2&1&0&0&1\\ 
1&2&2&1&0&0\\ 
\end{array} \right)
\end{equation*}

\begin{equation*}
\left( \begin{array}{rrrrrr}
 & & 3& -1&1&-1\\ 
1&1& 3&* & & \\
2& & & &*&*\\
2&-1& & &*&*\\
1&-1& & &*&* \\
 & &1.5&1.5&*&*  
\end{array} \right)
\left( \begin{array}{rrrrrr}
1/3& & & &1/2&2/3\\ 
-1/3& & & &1/2& \\ 
 & &1/3&2/3&1/6& \\ 
 & &1/3& &-1/6& \\ 
1/2&1&1/2& & & \\ 
1/2&1/3&-1/2& & & \\ 
\end{array} \right)
\end{equation*}
\fi

\begin{figure}[h]
\centering
\begin{tabular}{| c || c c c c |}
\hline
 & $\times^1$ & $\times^2$ & $\times^3$ & $\times^4$ \\
\hline
\hline
$\soc_1$ & 0.45 & 0.095 & 0.0023 & 0.0027 \\
$\soc_2$ & 0.069 &  0.021 & 0.0034 & 0.0033 \\
$\soc_3$ & 0.070 & 0.023 & 0.0036 & 0.0036 \\
$\soc_4$ & 0.071 & 0.022 & 0.0044 & 0.0033\\
\hline
\end{tabular}
\caption{Best error for regular $6$-gon}
\label{fig:numexp6gon}
\end{figure}

%The analytical solution for $\mathrm{SOCP}(1)$ is
%\begin{equation*}
%\left( \begin{array}{rrrrrr}
%2&2&2&2&2&2\\ 
%1&0&1&3&4&3\\ 
%1&0&1&3&4&3\\
%2&2&2&2&2&2\\ 
%3&4&3&1&0&1\\
%3&4&3&1&0&1\\
%\end{array} \right)
% = \left( \begin{array}{rr}
%1&\\ 
%1&-1\\ 
%1&-1\\ 
%1&\\ 
%1&1\\ 
%1&1\\ 
%\end{array} \right)
%\left( \begin{array}{rrrrrr}
%2&2&2&2&2&2\\ 
%1&2&1&-1&-2&-1\\ 
%\end{array} \right).
%\end{equation*}

\paragraph{$8$-gon.}  In our fourth example, we  factorize the slack matrix of the regular $8$-gon.  In Figure \ref{fig:numexp8gon}, we show the final residual errors obtained using our method.  We observe similar trends as we did with the previous instances and conjecture that the regular $8$-gon admits a $\times^l \soc_k$-lift if and only if $l \geq 4$.

\begin{figure}[h]
\centering
\begin{tabular}{| c || c c c c |}
\hline
 & 1 copy & 2 copies & 3 copies & 4 copies \\
\hline
\hline
$\soc_1$ & 0.44 & 0.073 & 0.029 & 0.0040 \\
$\soc_2$ & 0.038 &   0.028 & 0.010 & 0.0059 \\
$\soc_3$ & 0.040 & 0.027 & 0.0096 & 0.0068 \\
$\soc_4$ & 0.043 & 0.025 & 0.0093 & 0.0060 \\
\hline
\end{tabular}
\caption{Best error for regular $8$-gon}
\label{fig:numexp8gon}
\end{figure}
\paragraph{Acknowledgments.}  YS gratefully acknowledges Ministry of Education (Singapore) Academic Research Fund (Tier 1) R-146-000-329-133.  AV gratefully acknowledges Ministry of Education (Singapore) Start-Up Research Grant SRG ESD 2020 154 and NRF2019-NRF-ANR095 ALIAS grant. Both authors acknowledge useful discussions with  Professor Kim-Chuan Toh.

\bibliography{conic_MM,PSD_MM}

\begin{thebibliography}{10}

\bibitem{arora}
Sanjeev Arora, Rong Ge, Ravindran Kannan, and Ankur Moitra.
\newblock Computing a nonnegative matrix factorization -- provably.
\newblock In {\em STOC'2012}, pages 145--162, 2012.

\bibitem{BT}
Aharon Ben-Tal and Arkadi Nemirovski.
\newblock On polyhedral approximations of the second-order cone.
\newblock {\em Mathematics of Operations Research}, 26(2):193--205, 2001.

\bibitem{bhatia}
R.~Bhatia.
\newblock {\em Positive Definite Matrices}.
\newblock Princeton University Press, 2007.

\bibitem{cohen}
Joel~E. Cohen and Uriel~G. Rothblum.
\newblock Nonnegative ranks, decompositions, and factorizations of nonnegative
  matrices.
\newblock {\em Linear Algebra and its Applications}, 190:149--168, 1993.

\bibitem{FK}
J.~Faraut and A.~Kor{\'a}nyi.
\newblock {\em Analysis on Symmetric cones}.
\newblock Clarendon Press, 1994.

\bibitem{hamza}
H.~Fawzi.
\newblock On representing the positive semidefinite cone using the second-order
  cone.
\newblock {\em Mathematical Programming, Series A}, 175:109--118, 2019.

\bibitem{psdrank}
Hamza Fawzi, Jo{\~a}o Gouveia, Pablo~A. Parrilo, Richard~Z. Robinson, and
  Rekha~R. Thomas.
\newblock Positive {S}emidefinite {R}ank.
\newblock {\em Mathematical Programming}, 153(1):133--177, 2015.

\bibitem{Fay}
Leonid Faybusovich.
\newblock E. {}lieb convexity inequalities and noncommutative bernstein
  inequality in jordan-algebraic setting.
\newblock {\em Theoretical Mathematics and Applications}, 6(2):1--35, 2016.

\bibitem{sam}
Samuel Fiorini, Serge Massar, Sebastian Pokutta, Hans~Raj Tiwary, and Ronald
  de~Wolf.
\newblock Exponential {L}ower {B}ounds for {P}olytopes in {C}ombinatorial
  {O}ptimization.
\newblock {\em Journal of the ACM}, (17), 2015.

\bibitem{goemans}
M.~Goemans.
\newblock Smallest compact formulation for the permutahedron.
\newblock {\em Mathematical Programming}, 153(1):5--11, 2015.

\bibitem{lifts}
Jo{\~a}o Gouveia, Pablo~A. Parrilo, and Rekha Thomas.
\newblock Lifts of {C}onvex {S}ets and {C}one {F}actorizations.
\newblock {\em Mathematics of Operations Research}, 38(2):248--264, 2013.

\bibitem{gowda}
M.S. Gowda.
\newblock A h{\"o}lder type inequality and an interpolation theorem in
  euclidean jordan algebras.
\newblock https://arxiv.org/abs/1809.05417, 2018.

\bibitem{qcorr}
Rahul Jain, Yaoyun Shi, Zhaohui Wei, and Shengyu Zhang.
\newblock Efficient {P}rotocols for {G}enerating {B}ipartite {C}lassical
  {D}istributions and {Q}uantum {S}tates.
\newblock {\em IEEE Transctions on Information Theory}, 59:5171--5178, 2013.

\bibitem{LF20}
Dana Lahat and C\'edric F\'evotte.
\newblock Positive {S}emidefinite {M}atrix {F}actorization: {A} {L}ink to
  {P}hase {R}etrieval and a {B}lock {G}radient {A}lgorithm.
\newblock In {\em IEEE International Conference on Acoustics, Speech and Signal
  Processing}, 2020.

\bibitem{LF20b}
Dana Lahat and C\'edric F\'evotte.
\newblock Positive {S}emidefinite {M}atrix {F}actorization {B}ased on
  {T}runcated {W}irtinger {F}low.
\newblock In {\em 28th European Signal Processing Conference (EUSIPCO)}, 2020.

\bibitem{LLTF20}
Dana Lahat, Yanbin Lang, Vincent Y.~F. Tan, and C\'edric F\'evotte.
\newblock Positive {S}emidefinite {M}atrix {F}actorization: {A} {C}onnection
  with {P}hase {R}etrieval and {A}ffine {R}ank {M}inimization.
\newblock {\em IEEE Transactions on Signal Processing}, in press, 2021.

\bibitem{Lange}
Kenneth Lange.
\newblock {\em {M}{M} {O}ptimization {A}lgorithms}.
\newblock SIAM, 2016.

\bibitem{LSnature}
Daniel~D. Lee and H.~Sebastian Seung.
\newblock Learning the {P}arts of {O}bjects by {N}on-negative {M}atrix
  {F}actorization.
\newblock {\em Nature}, 401, 1999.

\bibitem{LS00}
Daniel~D. Lee and H.~Sebastian Seung.
\newblock Algorithms for {N}on-negative {M}atrix {F}actorization.
\newblock In {\em Advances in Neural Information Processing Systems 13}, 2000.

\bibitem{martin}
R.~Kipp Martin.
\newblock Using separation algorithms to generate mixed integer model
  reformulations.
\newblock {\em Oper. Res. Lett.}, 10(3):119--128, 1991.

\bibitem{SLG}
Alexander Schrijver, László Lovász, and Martin Grötschel.
\newblock {\em Geometric Algorithms and Combinatorial Optimization}.
\newblock Algorithms and Combinatorics. SpringerVerlag, 1993.

\bibitem{YSV}
Yong~Sheng Soh and Antonios Varvitsiotis.
\newblock A non-commutative extension of {L}ee-{S}eung's algorithm for positive
  semidefinite factorizations.
\newblock https://arxiv.org/abs/2106.00293.

\bibitem{sturm}
Jos~F. Sturm.
\newblock Similarity and other spectral relations for symmetric cones.
\newblock {\em Linear Algebra and its Applications}, 312(1--3):135--154, 2000.

\bibitem{VGG18}
Arnaud Vandaele, Fran{\c c}ois Glineur, and Nicolas Gillis.
\newblock Algorithms for {P}ositive {S}emidefinite {F}actorization.
\newblock {\em Computational Optimization and Applications}, 71(1):193--219,
  2018.

\bibitem{vanden}
L.~Vandenberghe.
\newblock Lecture notes on symmetric cones.
\newblock http://www.seas.ucla.edu/~vandenbe/236C/lectures/symmetric.pdf.

\bibitem{vavasis}
S.~Vavasis.
\newblock On the complexity of nonnegative matrix factorization.
\newblock {\em SIAM Journal on Optimization}, 20:1364--1377, 2009.

\bibitem{wong}
R.T. Wong.
\newblock Integer programming formulations of the traveling salesman problem.
\newblock In {\em IEEE International Conference on Circuits and Computers},
  pages 149--152, 1980.

\bibitem{yannakakis}
Mihalis Yannakakis.
\newblock Expressing {C}ombinatorial {O}ptimization {P}roblems by {L}inear
  {P}rograms.
\newblock {\em Journal of Computer and System Sciences}, 43:441--466., 1991.

\end{thebibliography}

\end{document}